\documentclass[oneside]{amsart}
\usepackage[colorlinks=true,allcolors=blue!80!black]{hyperref}
\usepackage{tikz-cd}
\usetikzlibrary{matrix, snakes, patterns, shadows.blur, backgrounds}

\usepackage{amssymb, amsthm}
\usepackage{enumerate, amsfonts, amsxtra,amsmath,latexsym,epsfig, color}
\usepackage{mathrsfs, MnSymbol}
\usepackage{geometry}                		
\usepackage{subcaption}
\usepackage{autobreak}
\usepackage{color}
\usepackage{amsmath,amscd}	
\usepackage{mathtools}
\usepackage{stmaryrd}
\usepackage{animate}
\usepackage{tikz-cd}
\usepackage{soul}
\usepackage{faktor}
\usepackage{comment}

\usepackage{textgreek}  

\newtheorem {theorem}{Theorem}[section]
\newtheorem {lemma} [theorem] {Lemma}
\newtheorem {proposition} [theorem] {Proposition}
\newtheorem {corollary} [theorem] {Corollary}

\theoremstyle{definition}

\newtheorem{remark}[theorem]{Remark}

\DeclareMathOperator{\PSL}{PSL}

\DeclareMathOperator{\lcm}{lcm}

\DeclareMathOperator{\Ind}{Ind}
\DeclareMathOperator{\Gal}{Gal}

\title{Corresponding Abelian Extensions of Integrally Equivalent Number Fields}
\author{Shaver Phagan}
\address{Department of Mathematics\\ Purdue University, West Lafayette, IN}
\email{phagan@purdue.edu}

\begin{document}

\maketitle

\section*{Abstract}  Extensive work has been done to determine necessary and sufficient conditions for a bijective correspondence of abelian extensions of number fields to force an isomorphism of the base fields.  However, explicit examples of correspondences over non-isomorphic fields are rare.  Integrally equivalent number fields admit an induced correspondence of abelian extensions.  Studying this correspondence using idelic class field theory and linear algebra, we show that the corresponding extensions share features similar to those of arithmetically equivalent fields, and yet they are not generally weakly Kronecker equivalent.  We also extend a group cohomological result of Arapura et. al. and present geometric and arithmetic applications.

\section{Introduction and Background}
\subsection{Summary}\hfill\\\\
\indent D. Prasad noted an isomorphism of the idele class groups of some non-isomorphic number fields \cite{prasad_refined_2017}, which induces a correspondence of their abelian extensions.  This article is principally concerned with determining the arithmetic similarity of extensions corresponding under this bijection.  In doing so, one is led to the notions of arithmetic, Kronecker, and weak Kronecker equivalence as means of comparing the splitting behavior of rational primes.  In Section \ref{sol}, we record the fact that a generic correspondence of abelian extensions of number fields cannot generally preserve arithmetic equivalence classes.  Then, in Section \ref{cae}, we prove that \textit{our} corresponding extensions are not even weakly Kronecker equivalent in general, yet they share features similar to those of arithmetically equivalent fields.  Some necessary technical facts on group schemes are recorded in Section \ref{tori}. In Section \ref{homology}, we conclude by extending a group cohomological result of Arapura et. al.  This is motivated by a geometric application to manifolds with contractible universal cover, but the techniques shed further light on the arithmetic correspondence.  We describe below some of the essential ingredients in this article, reviewing relevant facts and establishing context.
\subsection{Corresponding Abelian Extensions}\label{caeintro}\hfill\\\\

\indent A triple of groups \((G,G_1,G_2)\) is called an \textbf{integral Gassmann triple} if there is an isomorphism \(\mathbb{Z}[G/G_1]\simeq\mathbb{Z}[G/G_2]\) of \(\mathbb{Z}G\)-modules, where the \(G\)-action is permutation of cosets induced by left multiplication.  Let \(K/\mathbb{Q}\) be a Galois extension of number fields with group \(G\).  For a subfield \(F\) of \(K\), let \(G_F\) be the subgroup of \(G\) fixing \(F\).  Subfields \(K_1\) and \(K_2\) of \(K\) are called \textbf{integrally equivalent} if \((G,G_{K_1},G_{K_2})\) is an integral Gassmann triple.  Integral equivalence of number fields \(K_1\) and \(K_2\) induces an isomorphism \(C_{K_1}\simeq C_{K_2}\) of idele class groups, and this gives a correspondence of their abelian extensions, by class field theory (c.f. Section \ref{cae}).  Letting \(\Delta_F\) be the discriminant of a number field \(F\), and \(K_n(F)\) its \(n^{th}\) \(K\)-group (c.f. Section \ref{Kgroups}), corresponding abelian extensions of integrally equivalent number fields enjoy the following relations.

\begin{theorem}\label{P}
    If \(L_i/K_i\), \(i=1,2\) are corresponding abelian extensions of integrally equivalent number fields, then, for \(m=[L_i:K_i]\),
    \begin{enumerate}
    \item \(L_1,L_2\) have the same degree, Galois closure, and maximal abelian subextension over \(\mathbb{Q}\)
    \item \(L_1\) and \(L_2\) have the same signature for \(m\) odd
    \item \(K_n(L_1)\simeq K_n(L_2)\) for odd \(n\geq 3\) when \(m\) is odd
    \item \(\Delta_{L_1}\) and \(\Delta_{L_2}\) have the same prime divisors
    \item\label{zetarels} There is a set \(\mathcal{P}\) of rational primes of positive density, so that if \(\zeta_{L_i}(s)=\sum_{\mathbb{N}_+}\frac{a_i(n)}{n}\), and every prime divisor of \(n\) is in  \(\mathcal{P}\), then \(a_1(n)=a_2(n)\) (c.f. Corollary \ref{zeta}).  Furthermore, the \(L_i\) are ultra-coarsely arithmetically equivalent (c.f. Section \ref{splittingtype}).
    \end{enumerate}
\end{theorem}

\noindent Readers acquainted with the topic of \textit{arithmetic similarity} (c.f. Section \ref{splittingtype}, and \cite{perlis_equation_1977, klingen_zahlkorper_1978, jehne_kronecker_1977, lochter_new_1995, lochter_weakly_1994, lochter_weak, klingen_arithmetical_1998}) might have already noted the similarity of Theorem \ref{P} to the following.

\begin{theorem}\cite{perlis_equation_1977}\label{PP}
    If \(F_1,F_2\) are arithmetically equivalent number fields, then
    \begin{enumerate}
    \item \(F_1,F_2\) have the same degree, Galois closure, and maximal Galois subextension over \(\mathbb{Q}\)
    \item \(F_1\) and \(F_2\) have the same signature
    \item \(K_n(F_1)\simeq K_n(F_2)\) for odd \(n\geq 3\)
    \item \(\Delta_{F_1}=\Delta_{F_2}\)
    \item \(\zeta_{F_1}=\zeta_{F_2}\).
    \end{enumerate}
\end{theorem}

\noindent This resemblance between arithmetically equivalent number fields and corresponding abelian extensions of integrally equivalent number fields is especially striking, since we show in Theorem \ref{notwKeq} that the corresponding extensions are not even weakly Kronecker equivalent in general.  On the other hand, integral equivalence appears to be a significantly stronger condition than arithmetic equivalence, and, since it induces a correspondence of abelian extensions, one might guess that certain qualities of arithmetic similarity are inherited.\\\\  
\noindent A triple of groups \((G,G_1,G_2)\) is called a \textbf{Gassmann triple} if there is an isomorphism of \(\mathbb{Q}G\)-modules \(\mathbb{Q}[G/G_1]\simeq\mathbb{Q}[G/G_2]\).  Number fields \(F_1\) and \(F_2\) are \textbf{arithmetically equivalent} if and only if \((G,G_{F_1},G_{F_2})\) is a Gassmann triple, where \(G\) is the group of the Galois closure of the \(F_i\), \(i=1,2\), over the rational numbers.  From the perspective of group (co)homology, the constituents \(G_1,G_2\) of an integral Gassmann triple \((G,G_1,G_2)\) enjoy equality of (co)homology groups with trivial action on torsion coefficients, but this is not always so for an ordinary Gassmann triple \((G,G_1,G_2)\) (c.f. Corollary \ref{blemma} in this article, and \cite{bartel_torsion_2016}).  Furthermore, there are several constructions which produce Gassmann triples (c.f. \cite{perlis_equation_1977, sunada_riemannian_1985}), but all known examples of integral Gassmann triples fundamentally arise from a \textit{nuclear} triple, originally due to L. Scott \cite{scott_integral_1993}, and whose existence appears to the author a miracle.  Scott noticed that \(G=\PSL(2,\mathbb{F}_{29})\) has non-conjugate subgroups \(G_1\) and \(G_2\), both isomorphic to \(A_5\), such that \((G,G_1,G_2)\) is an integral Gassmann triple.  It was D. Zywina who then showed that \(G\) is the Galois group of an extension of number fields \cite{zywina_inverse_2015},  and D. Prasad who first noted that this implies the isomorphism of idele class groups of what we are calling integrally equivalent number fields \cite{prasad_refined_2017}. We recall some history motivating our study of the induced correspondence of abelian extensions.\\\\

\noindent For number fields \(K_1,K_2\) with absolute Galois groups \(\Gamma_1,\Gamma_2,\) respectively, the celebrated Neukirch-Uchida Theorem \cite{uchida_isomorphisms_1976} says that \(\Gamma_1\simeq\Gamma_2\) if and only if \(K_1\simeq K_2\), but it is known that an isomorphism of abelianized Galois groups \(\Gamma_1^{ab}\simeq\Gamma_2^{ab}\) does not imply an isomorphism \(K_1\simeq K_2\).  Indeed, \(\Gamma_1^{ab}\simeq\Gamma_2^{ab}\) for non-isomorphic integrally equivalent \(K_1,K_2\), since \(\Gamma_{i}^{ab}\simeq \widehat{C_{K_i}}\), where \(C_{K_i}\) is the idele class group, and the hat indicates profinite completion.  Now, Cornelissen et. al. have devoted considerable effort to reconstructing global fields from their abelianized Galois groups and related data (c.f. \cite{cornelissen_curves_2013, cornelissen_reconstructing_2017, ancona_quantum_2014}), culminating in \cite{cornelissen_characterization_2018}, whose main theorem says the existence of an isomorphism \(\Gamma_1^{ab}\simeq\Gamma_2^{ab}\) inducing an \(L\)-function preserving correspondence of Dirichlet characters is equivalent to an isomorphism \(K_1\simeq K_2\), for global fields \(K_1\) and \(K_2\).  Furthermore, they prove the following theorem in the number field case.
\begin{theorem}\label{dschar}
\cite{cornelissen_characterization_2018} Given a number field \(K\) and an integer \(m\geq 3\), there is a character \(\chi\) of \(\Gamma_K^{ab}\) of order \(m\) such that if \(F\) is a number field admitting a character \(\chi'\) with \(L_K(\chi)=L_F(\chi')\), then \(F\simeq K\). 
\end{theorem}
\noindent As a corollary, one can deduce that a number field is uniquely determined up to isomorphism by the set of Dedekind zeta functions of its finite abelian extensions, as observed by Solomatin in \cite{solomatin_note_2019}. We recall the argument below for the convenience of the reader (c.f. Section \ref{sol}).  This result suggests asking how arithmetically similar corresponding abelian extensions can be when the base fields are non-isomorphic.  Despite the cited work of Cornelissen et. al., \textit{examples} of correspondences of abelian extensions do not abound in the literature, so this converse question has not received much attention.  We explore it here and provide an answer with Theorems \ref{P} and \ref{notwKeq}.\\

\noindent Some final remarks on notation and terminology.  Throughout the sequel, the symbol \(\prod\) is used to indicate a direct product, except in the proofs of Theorem \ref{idelenorm} and Corollary \ref{zeta}.  Furthermore, an integral Gassmann triple \((\Gamma,\Gamma_1,\Gamma_2)\) is assumed to have \([\Gamma:\Gamma_i]<\infty\), and number fields \(K_1,K_2\) will be called \textbf{integrally equivalent over} \(F\) if \((G,G_{K_1},G_{K_2})\) is an integral Gassmann triple, where \(G\) is the group of the Galois closure of \(K_i/F\), \(i=1,2\).  Lastly, \(0\in\mathbb{N}\), and \(\mathbb{N}_+=\mathbb{N}-\{0\}\).

\subsection{Comparing Splitting Types}\label{splittingtype}\hfill\\\\
\indent Let \(p\mathcal{O}_K=\mathfrak{p}_1^{e_1}\cdots\mathfrak{p}_n^{e_n}\) be the decomposition of the rational prime \(p\) into prime ideals in the ring of integers of a number field \(K\).  Define the \textbf{splitting type} \(S_K(p)\) as the multiset of residues of \(K\) over \(p\), that is \(S_K(p)=\{\{f_i\}\}_{i=1}^n\), with \(f_i=[\mathcal{O}_K/\mathfrak{p}_i:\mathbb{Z}/p\mathbb{Z}]\).  Arithmetic equivalence of number fields \(F_1,F_2\) can be characterized as an equality \(S_{F_1}(p)=S_{F_2}(p)\) for all but finitely many rational primes \(p\). Call \(F_1,F_2\) \textbf{Kronecker equivalent} if \(1\in S_{F_1}(p) \text{ if and only if } 1\in S_{F_2}(p)\) for all but finitely many \(p\), and \textbf{weakly Kronecker equivalent} if \(\text{gcd}(S_{F_1}(p))=\text{gcd}(S_{F_2}(p))\) for all but finitely many \(p\).  We record a useful theorem for reference later.
\begin{theorem}[\cite{lochter_weak} Theorem \(3'\) ]\label{noexwKreq}
 If \(F_1\) and \(F_2\) are weakly Kronecker equivalent number fields, then for any prime \(p\), \(\gcd(S_{F_1}(p))=1\) if and only if \(\gcd(S_{F_2}(p))=1\).
\end{theorem}
\noindent In fact, \(F_1\) and \(F_2\) are Kronecker equivalent if and only if \(\mathcal{N}_{F_1,p}=\mathcal{N}_{F_2,p}\) for all \(p\), where \(\mathcal{N}_{F_i,p}=\{\sum_j n_jf_j\text{ }|\text{ }n_j\in\mathbb{N},f_j\in S_{F_i}(p)\}\) \cite{lochter_new_1995}.  In particular, Kronecker equivalence implies weak Kronecker equivalence, and, clearly, arithmetic equivalence implies Kronecker equivalence. Furthermore, for Kronecker equivalent \(F_1,F_2\), if \(\zeta_{F_i}(s)=\sum_{n\in\mathbb{N}_+}\frac{a_i(n)}{n^s}\), then \(a_1(n)\neq 0\) if and only if \(a_2(n)\neq 0\), so that Kronecker equivalence is a sort of \textit{approximate arithmetic equivalence}.  Weak Kronecker equivalence does not support a similar interpretation in a nice way, but a related notion does. We call number fields \(F_1\) and \(F_2\) such that \(\lcm(S_{F_1}(p))=\lcm(S_{F_2}(p))\) for all but finitely many \(p\) \textbf{ultra-coarsely arithmetically equivalent}.  This terminology is motivated by Proposition \ref{ultra-coarse}.  As it turns out, number fields are ultra-coarsely arithmetically equivalent if and only if they have the same Galois closure. 
\begin{proposition}\label{uclemma}
    Number fields \(K_1,K_2\) are ultra-coarsely arithmetically equivalent if and only if they have the same Galois closure over \(\mathbb{Q}\).
\end{proposition}
\begin{proof}
    If \(K_1,K_2\) are ultra-coarsely arithmetically equivalent number fields, a rational prime splits completely in \(K_1\) if and only if it does so in \(K_2\), with at most finitely many exceptions, so the Galois closure of \(K_1/\mathbb{Q}\) is equal to that of \(K_2/\mathbb{Q}\).  Now, suppose \(K_1,K_2\) are number fields with the same Galois closure \(K/\mathbb{Q}\), and let \(G=\Gal(K/\mathbb{Q})\). Letting \(\lambda_{i,p}=\lcm(S_{K_i}(p))\), if \(p\) is a rational prime unramified in \(K\) with Frobenius class \(F_p\subset G\), we have \(\lambda_{i,p}|n\) if and only if \(F_p^n\subset G_{K_i}\), which is equivalent to \(F_p^n=1\), since \(K\) is the Galois closure of \(K_i\), and \(F_p^n\) is a \(G\)-conjugacy class.  In particular, \(\lambda_{1,p}=o(F_p)=\lambda_{2,p}\).
\end{proof}

\begin{proposition}\label{ultra-coarse}
    Suppose \(F_1,F_2\) are ultra-coarsely arithmetically equivalent number fields, and let \(S\) be the finite set of rational primes \(p\) such that \(\lcm(S_{F_1}(p))\neq\lcm(S_{F_2}(p))\).  If \(\zeta_{F_i}(s)=\sum_{n\in\mathbb{N}_+}\frac{a_i(n)}{n^s}\), and \(n\) is a positive integer with no prime divisor in \(S\), we may write \(\lambda_p=\lcm(S_{F_i}(p))\) unambiguously for \(p\) dividing \(n\), and, furthermore, if \(a_1(n)\neq 0\), there is \(m\in\mathbb{N}_+\) with \(a_2(m)\neq 0\) and \(n\Lambda^{-1}<m<n\Lambda\), where \(\Lambda = p_1^{\lambda_{p_1}}\cdots p_k^{\lambda_{p_k}}\) and \(p_1,...,p_k\) are the prime divisors of \(n\).
\end{proposition}
\begin{proof}
    Given a prime \(p\) not in \(S\), we know \(\lambda_p\mathbb{N}_+\subset\mathcal{N}_{F_i,p}\), \(i=1,2\), so if \(c_p\in\mathcal{N}_{F_1,p}\), there is \(d_p\in\mathcal{N}_{F_2,p}\) such that \(|c_p-d_p|<\lambda_p\).  If \(n=p_1^{c_1}\cdots p_{k}^{c_k}\), we know that \(c_j\in\mathcal{N}_{F_1,p_j}\), so there is \(d_j\in\mathcal{N}_{F_2,p_j}\) such that \(|d_j-c_j|<\lambda_{p_j}\), for \(1\leq j\leq k\).  Setting \(m=p_1^{d_1}\cdots p_k^{d_k}\), we know that \(m\) is the absolute norm of some ideal in \(\mathcal{O}_{F_2}\), so that \(a_2(m)\neq 0\), and furthermore, \(\frac{m}{n}=p_1^{d_1-c_1}\cdots p_k^{d_k-c_k}\), so that \(\Lambda^{-1}<\frac{m}{n}<\Lambda\).
\end{proof}

\subsection{\(K\)-Groups of Odd Index}\label{Kgroups}\hfill\\\\
\indent The \(K\)-theory of a number field \(F\) is a more contemporary invariant. We will restrict our attention here to groups \(K_n(F)\) with \(n\) odd.  We show that arithmetically equivalent fields have the same \(K\)-groups with odd \(n\geq 3\) and record some useful facts about the groups \(K_n(F)\) for the convenience of the reader.  More details can be found in Section 5.3 of \cite{weibel_algebraic_nodate} or Chapter VI of \cite{weibel_k-book_2013}.\\

\noindent The \(K\)-groups are abelian, and \(K_1(F)=F^\times\).  For odd \(n\geq 3\), \(K_n(F)\) is a finitely generated group, given by the following rule.
$$
K_n(F)\simeq\begin{cases}
    \mathbb{Z}^{r_1+r_2}\bigoplus\mathbb{Z}/w_i\mathbb{Z}, & n\equiv 1 \text{ mod } 8\\
    \mathbb{Z}^{r_2}\bigoplus(\mathbb{Z}/2\mathbb{Z})^{r_1-1}\bigoplus\mathbb{Z}/2w_i\mathbb{Z}, & n\equiv 3 \text{ mod } 8 \\
    \mathbb{Z}^{r_1+r_2}\bigoplus\mathbb{Z}/\frac{1}{2}w_i\mathbb{Z}, & n\equiv 5 \text{ mod } 8\\
    \mathbb{Z}^{r_2}\bigoplus\mathbb{Z}/w_i\mathbb{Z}, & n\equiv 7 \text{ mod } 8
\end{cases}
$$
where \(r_1\) \((r_2)\) is the number of real (complex) places of \(F\), \(i=(n+1)/2\), and
\begin{equation}\label{exp}
    v_p(w_i)=\text{max}\{\nu\text{ }|\text{ Gal}(F(\zeta_{p^\nu})/F)\text{ has exponent dividing }i\},
\end{equation}
 \(\zeta_r\) is a primitive \(r^{th}\) root of unity, and \(v_p\) is the \(p\)-adic valuation associated to the rational prime \(p\).

\begin{proposition}\label{qgivesk}
    If \(F_1,F_2\) are arithmetically equivalent number fields, then for odd \(n\geq 3\), we have \(K_n(F_1)\simeq K_n(F_2)\).  
\end{proposition}
\begin{proof}
Since the \(F_i\) are arithmetically equivalent, we know they have the same signature.  Furthermore, given a primitive root of unity \(\mu\), we know that \(F_1\cap\mathbb{Q}(\mu)=F_2\cap\mathbb{Q}(\mu)\), so \(\text{Gal}(F_1(\mu)/F_1)\simeq\text{Gal}(F_2(\mu)/F_2)\), since \(\text{Gal}(F_i(\mu)/F_i)\simeq\text{Gal}(\mathbb{Q}(\mu)/F_i\cap\mathbb{Q}(\mu))\).
\end{proof}

\subsection{Diagrams in (Co)Homology and Corresponding Abelian Covers}\hfill\\\\
\indent We say that manifolds \(M_1,M_2\) are \textbf{integrally equivalent} if there are coverings \(M_1,M_2\rightarrow M\) with normal closure \(N\rightarrow M\) a \(G\)-covering, where \((G,G_1,G_2)\) is an integral Gassmann triple, and \(M_i\) is the subcover corresponding to \(G_i\). Arapura et. al. construct non-isometric integrally equivalent closed hyperbolic manifolds \cite{arapura_integral_2019}.  Their results imply that for integrally equivalent manifolds \(M_1,M_2\) with contractible universal cover, there are isomorphisms in cohomology \(H^*(M_1)\simeq H^*(M_2)\) natural with respect to restriction (corestriction) from (to) \(H^*(M)\).  These already suggest a correspondence of the abelian covers when the manifolds are closed and orientable, by Poincare duality.  We show that there are similar isomorphisms in homology and define a correspondence of the abelian covers of integrally equivalent manifolds \(M_1\) and \(M_2\), as follows.  Covers \(M_i'\rightarrow M_i\) \textbf{correspond} if \(U_i\subset H_1(M_i)\), \(i=1,2\) are finite index subgroups such that the isomorphism \(H_1(M_1)\simeq H_1(M_2)\) restricts to an isomorphism \(U_1\simeq U_2\), and \(\pi_1(M_i')\) is the preimage of \(U_i\) under the projection \(\pi_1(M_i)\rightarrow H_1(M_i)\). The diagrams in homology and cohomology are obtained using group (co)homology arguments \footnotemark\footnotetext{We pursue group (co)homological relations afforded by integral Gassmann triples insofar as they elucidate our geometric correspondence, but there is much more to be said.  Homological relations in the spirit of Stallings' Theorem and applications thereof are explored in \cite{golich_diamonds_2023}.} \footnotemark\footnotetext{A conceivable approach to our study of corresponding abelian extensions would be to lead with the group (co)homological results from Section \ref{homology} and then define our correspondence using the isomorphism in first homology from  Corollary \ref{blemma}, with \((\Gamma, \Gamma_1,\Gamma_2)\) an integral Gassmann triple of absolute Galois groups of number fields.  As we will see, the adelic approach taken in Section \ref{cae} allows us to easily circumvent the coprimality assumptions needed in Theorem \ref{gthm} (c.f. Theorem \ref{galoisclosure}).} in Section \ref{homology}, and the more geometrically flavored Theorem \ref{g1} is obtained as a quick corollary.
\begin{theorem}\label{g1}
    If \(M\) is a manifold with contractible universal cover, and \(M_1,M_2\rightarrow M\) are integrally equivalent degree \(d\) covers, then corresponding abelian covers \(M_i'\rightarrow M_i\), \(i=1,2\) of degree \(d'\) have the same normal closure over \(M\) when \((d,d')=1\).
\end{theorem}

\subsection*{Acknowledgements}
I would like to thank my advisor Ben McReynolds for all his help, and I would like to thank Daniel Le for his helpful comments and a clarifying conversation about algebraic tori.  I would also like to thank Milana Golich and Justin Katz for conversation on this work.  Thanks also to Zachary Selk and Dustin Lee Enyeart for helpful comments that improved the quality of exposition.

\section{Solomatin's Theorem}\label{sol}
\indent This section is devoted to proving the following theorem.  We take liberties with certain details of the exposition and proof but mostly follow \cite{solomatin_note_2019}.  Given a group \(C\), we use \(\check{C}\) to denote the character group of homomorphisms \(C\rightarrow \mathbb{C^\times}\).
\begin{theorem}\label{solthm}
    \cite{solomatin_note_2019} If \(Z_F\) is the set of Dedekind zeta functions of finite abelian extensions of a number field \(F\), then \(Z_{K_1}=Z_{K_2}\) if and only if \(K_1\simeq K_2\).
\end{theorem}
\noindent Note that \(Z_{K_1}=Z_{K_2}\) implies \(\zeta_{K_1}=\zeta_{K_2}\).  Indeed, there is an abelian extension \(L_2/K_2\) such that \(\zeta_{L_2}=\zeta_{K_1}\), but then \(\zeta_{K_2}=\zeta_{L_1}\) for some abelian extension \(L_1/K_1\), so that $$[L_2:\mathbb{Q}]=[K_1:\mathbb{Q}]\leq [L_1:\mathbb{Q}]=[K_2:\mathbb{Q}],$$ and therefore \(L_2=K_2\). Before proving Theorem \ref{solthm}, we review preliminary results and set terminology.\\\\
\noindent Given a finite cyclic group \(C\) and a Galois extension \(K/\mathbb{Q}\) of number fields with subfield \(F\subset K\), letting \(G=\Gal(K/\mathbb{Q})\), we call \(K'\) a \textbf{wreathing extension of} \(K/F/\mathbb{Q}\) \textbf{by} \(C\) if the following four conditions are satisfied. 
\begin{enumerate}
    \item \(K\subset K'\)
    \item \(K'/\mathbb{Q}\) is Galois with group \(\Gal(K'/\mathbb{Q})=C[G/G_F]\rtimes G\)
    \item \(\Gal(K'/F)=C[G/G_F]\rtimes G_F\)
    \item \(\Gal(K'/K)=C^{[G:G_F]}\)
\end{enumerate} 
where the action of the semi-direct product is permutation of cosets according to left multiplication.  When knowledge of \(K\) is either implicit or unnecessary, we refer to \(K'\) as a wreathing extension of \(F\) by \(C\).
\begin{theorem}\cite{cornelissen_characterization_2018}
Given a finite Galois extension \(K/\mathbb{Q}\) of number fields, a subfield \(F\subset K\), and a finite cyclic group \(C\), there is a wreathing extension of \(K/F/\mathbb{Q}\) by \(C\).
\end{theorem}
\noindent Given a wreathing extension \(K'\) of \(K/F/\mathbb{Q}\) by \(C\), set \(G'=\Gal(K'/\mathbb{Q})\), so that \(G'_F=C[G/G_F]\rtimes G_F\), and arrange so that the first coordinate of \(C[G/G_F]\) is fixed by \(G_F\) (i.e. the first coordinate corresponds to the coset \(G_F\in G/G_F\)).  Letting \(n=[G:G_F]\), to \(K'\) we associate the morphism
$$
\chi_{K'}: G'_F\rightarrow C:(c_1,...,c_n,g)\mapsto c_1.
$$
If \(C=\mathbb{Z}/m\mathbb{Z}\), a homomorphism \(\varphi: C\rightarrow\{e^{2\pi ik/m}\}_{k=0}^{m-1}\subset\mathbb{C}^\times\) determines a one-dimensional representation \(\varphi\circ\chi_{K'}\) of \(G'_F\), which we denote \(\chi_\varphi\).  Note that each \(\chi_\varphi\) gives a Dirichlet character of \(\Gamma_F\) via the composition 
$$
\Gamma_F\rightarrow G_F'/\ker(\chi_\varphi)\xrightarrow{\overline{\chi_\varphi}}\mathbb{C}^\times,
$$
which we also refer to as \(\chi_\varphi\).  We will further abuse notation and identify \(\chi_{K'}\) with \(\chi_{\varphi'}\), where \(\varphi'(k)=e^{2\pi i k/m}\). Two lemmas are needed for the proof of the theorem.
\begin{lemma}\label{dscharacter}
    If \(K'\) is a wreathing extension of a number field \(F\) by a finite cyclic group of order greater than \(2\), and \(F'\) is a number field admitting a Dirichlet character \(\chi'\) satisfying \(L_{F'}(\chi')=L_F(\chi_{K'})\), then \(F\simeq F'\).
\end{lemma}
\begin{proof}
    Follows from the proof of Theorem 10.1 in \cite{cornelissen_characterization_2018}.
\end{proof}
\begin{lemma}\label{strongsollemma}
If \(K'\) is a wreathing extension of a number field \(F\) by a finite cyclic group \(C\), and \(\varphi\in\check{C}\) is non-trivial, then the induced representation \(V = \Ind_{G'_F}^{G'}(\chi_\varphi)\) is irreducible.
\end{lemma}
\begin{proof}
    We use \(e\) to denote the identity element of \(G\).  Let \(g_1,...,g_n\in G\) be a complete set of representatives for \(G/G_F\), and assume without loss of generality that \(g_1\in G_F\).  Then \(g_i'=(0,g_i)\), \(i=1,...,n\), is a complete set of coset representatives for \(G'/G'_F\), and  \(V\) decomposes as
    $$
    V = \bigoplus\limits_{i=1}^n g'_i\mathbb{C}.
    $$
    We argue that the action of \(\mathbb{C}G'\) on \(V\) is transitive by showing \(g'_1\in\mathbb{C}G'.v\) for \(0\neq v\in V\).  Recall that we arrange so that the first coordinate of \(C[G/G_F]\) is fixed by \(G_F\), and let \(\lambda = \chi_\varphi(1)\). Note that if \(v\neq 0\), there is \(g\in G'\) such that the coefficient of \(g'_1\) in \(g.v\) is non-zero, so it suffices to prove the claim under the assumption \(v=\sum_ig'_ic_i\) with \(c_1\neq 0\).  Under this assumption, let \(v'=(1,0,...,0,e).v\).  Compute \(v' = \lambda g'_1c_1+\sum_{i=2}^{n}g'_ic_i\), so that \((v-v') = (1-\lambda)g'_1c_1\).  Since \(\varphi\) is not trivial, we know \(\lambda\neq 1\), and hence \(g'_1\in \mathbb{C}G'.v\).
\end{proof}

\noindent We may now prove Solomatin's Theorem.  Any properties of Artin \(L\)-functions used below can be found in Lemma 10.2 in \cite{cornelissen_characterization_2018}.
\begin{proof}[\textbf{Proof of Theorem \ref{solthm}}]
    We need only show that \(Z_{K_1}=Z_{K_2}\) implies \(K_1\simeq K_2\), as the converse is obvious.  Let \(K_1\) be a number field and \(K'\) a wreathing extension of \(K_1\) by a cyclic group \(C_1\) of order \(m\geq 3\), and take \(L_1/K_1\) to be the abelian extension given by \(G'_{L_1}=\ker(\chi_{K'})\), where \(G'=\Gal(K'/\mathbb{Q})\).  Observe that \(\Gal(L_1/K_1)\simeq C_1\) via \(\chi_{K'}\). Now suppose \(Z_{K_1}=Z_{K_2}\) for a number field \(K_2\), so that \(\zeta_{K_1}=\zeta_{K_2}\) and \(\zeta_{L_1}=\zeta_{L_2}\) for some abelian extension \(L_2/K_2\). Letting \(1\) denote the trivial representation and \(\rho_{F}=\Ind_{G'_F}^{G'}\),  we know that $$\zeta_{K_i}=L_{K_i}(1)=L_{\mathbb{Q}}(\rho_{K_i}(1)),$$
    so that \(L_{\mathbb{Q}}\left(\rho_{K_1}(1)\right)= L_{\mathbb{Q}}\left(\rho_{K_2}(1)\right)\) and hence 
    \begin{equation}\label{lfuncK}
    \rho_{K_1}(1)\simeq\rho_{K_2}(1).
    \end{equation}
    Similarly, \(\zeta_{L_1}=\zeta_{L_2}\) implies 
    \begin{equation}\label{lfuncL}
    \rho_{L_1}(1)\simeq\rho_{L_2}(1).
    \end{equation}
    Now, \(\rho_{L_i}(1)=\rho_{K_i}(\Lambda_{L_i/K_i})\), where \(\Lambda_{L_i/K_i}\) is the permutation representation given by the left multiplication action of \(G_{K_i}'\) on \(G_{K_i}'/G_{L_i}'\), which factors through the left regular representation of \(\Gal(L_i/K_i)\).  Letting \(C_2=\Gal(L_2/K_2)\), equation \ref{lfuncL} can therefore be rewritten
    $$
    \bigoplus_{\varphi\in\check{C}_1}\rho_{K_1}(\chi_\varphi)\simeq\bigoplus_{\psi\in\check{C}_2}\rho_{K_2}(\psi'),
    $$
    which, along with equation \ref{lfuncK} implies
    \begin{equation}\label{decomp}\bigoplus_{1\neq\varphi\in\check{C}_1}\rho_{K_1}(\chi_\varphi)\simeq\bigoplus_{1\neq\psi\in\check{C}_2}\rho_{K_2}(\psi'),
    \end{equation}
    where \(\psi'\) is the character \(G'_{K_2}\rightarrow\mathbb{C}^\times\) induced by \(\psi\in\check{C_2}\). Now, one has \(|C_2|=|C_1|=m\), so that either side of equation \ref{decomp} has \(m-1\) direct summands, and by Lemma \ref{strongsollemma}, each summand on the left hand side is an irreducible representation, so that the right hand side of equation \ref{decomp} also consists of irreducible representations.  Since \(\rho_{K_1}(\chi_{K'})\) is a summand of the left hand side of equation \ref{decomp}, we conclude there is \(\psi\in\check{C}_2\) such that \(\rho_{K_1}(\chi_{K'})\simeq\rho_{K_2}(\psi')\).  But then 
    $$
    L_{K_1}(\chi_{K'})=L_\mathbb{Q}(\rho_{K_1}(\chi_{K'}))=L_\mathbb{Q}(\rho_{K_2}(\psi'))=L_{K_2}(\psi'),
    $$ 
    so that \(K_1\simeq K_2\), by Lemma \ref{dscharacter}.
\end{proof}

\section{Quasi-split Tori and Idele Norm}\label{tori}

In this section, we prove that taking adelic points of a certain group scheme recovers the idele norm.  This result is likely known to experts, but for the sake of completeness and for lack of a suitable reference, we include a proof.  We write \(K^s\) for the separable closure of a field \(K\), and \(\mathbb{G}_m\) will denote the multiplicative group scheme.  We are interested in relating Weyl restriction of \(\mathbb{G}_m\) to the idele norm.  For further discussion of the lemmas, see Chapter 2 in \cite{platonov_algebraic_1994}.  Given a subset \(S\) of an abelian group, we will use the notation \(\sum S\) to denote the sum \(\sum_{s\in S}s\).  We will also write \(\gamma^g\) in place of \(g^{-1}\gamma g,\) given elements \(\gamma,g\) of a group \(G\).\\

\noindent  Let \(K\) be  a local or global field.  It is well-known that finite-dimensional \(K\)-tori correspond to finitely generated free abelian groups, equipped with an action of \(\mathcal{G}=\text{Gal}(K^s/K)\), via a contravariant equivalence of categories.  In particular, the group-module corresponding to a \(K\)-torus \(T\) is its character group \(\text{Hom}(T_{K^s},\mathbb{G}_m)\).  Since \(T\) is finite-dimensional over \(K\), it splits over a finite Galois extension \(E/K\).  Furthermore, if \(T\) is a so-called \textit{quasi-split} torus and \(G=\text{Gal}(E/K)\), there is a finite \(G\)-set \(A\) and a \(G\)-equivariant isomorphism \(T_{K^s}\simeq\mathbb{G}_m^A\), where the right hand side is a direct product of copies of \(\mathbb{G}_m\), indexed by \(A\) with \(G\)-action given by permutation of coordinates.  This induces an action of \(\mathcal{G}\) on the character group.  The set of quasi-split tori is precisely the set of products of Weyl restrictions of \(\mathbb{G}_m\) (c.f. Chapter 2 in \cite{platonov_algebraic_1994}).  A straightforward adaptation of the proof of Theorem 7.5  in \cite{waterhouse_introduction_1979} yields the following. 
\begin{lemma}\label{chargroup}
    Given \(E/K\) Galois with group \(G\) and subextension \(K'/K\), letting, \(\Omega=G/G_{K'}\) and \(T'=\text{Res}_{K'/K}(\mathbb{G}_m)\), the character group of \(T'\) is \(\mathbb{Z}\Omega\).
\end{lemma}
\noindent We will need an additional observation before proving the main theorem.  Let \(\alpha=\sum \Omega\), and let \(N: T'\rightarrow\mathbb{G}_m\) be the \(K\)-scheme corresponding to the morphism \(\iota:\mathbb{Z}\rightarrow\mathbb{Z}\Omega:1\mapsto\alpha\).  The \(K^s\)-form of \(N\) has associated Hopf algebra morphism 
$$N^*_{K^s}:K^s[X,X^{-1}]\rightarrow K^s[X_i,X_i^{-1}]:X\mapsto X_1\cdots X_n,$$ where \(X_i\) is the coordinate given by \(\overline{g}_i\), and \(g_1,...,g_n\) is a full set of coset representatives for \(G/G_{K'}\), so that taking \(K\)-points of \(N\) recovers the usual field norm.  In summary:
\begin{lemma} \label{norm}
    The \(K\)-scheme \(N\) corresponding to the inclusion of modules \(\iota\) has \(K\)-points given by the field norm \(N(K)=N_{K'/K}:(K')^\times\rightarrow K^\times\). 
\end{lemma}
\noindent With the lemmas recorded, we are ready to state and prove the theorem.
\begin{theorem}\label{idelenorm}
    Let \(E/K\) be a Galois extension of number fields with group \(G\), and suppose \(G_{K'}\subset G\) is the stabilizer of the subfield \(K'\subset E\), \(\Omega=G/G_{K'}\), and \(\alpha = \sum\Omega\).  If \(\iota\) is the morphism \(\mathbb{Z}\rightarrow\mathbb{Z}\Omega\) of \(G\)-modules given by \(1\mapsto\alpha\), and \(N\) is the scheme corresponding to \(\iota\), then \(N(\mathbb{A}_K):\mathbb{I}_{K'}\rightarrow\mathbb{I}_K\) is the idele norm.
\end{theorem}
\begin{proof}
    It suffices to check componentwise, so we need to compute \(N(K_\nu)\) for a place \(\nu\) of \(K\).  To do so, we view \(\mathbb{Z}\Omega\) as a \(\mathcal{D}_{\omega/\nu}\)-module, where \(\omega\) is a place of \(E\) over \(\nu\), and \(\mathcal{D}_{\omega/\nu}\subset G\) is the associated decomposition group.  Let \(g_1,...,g_m\) be a complete set of representatives in \(G\) for the double cosets \(\mathcal{D}_{\omega/\nu}\backslash G/G_{K'}\).  If \(\Omega_i\) is the \(\mathcal{D}_{\omega/\nu}\)-orbit of \(\overline{g}_i\) in \(\Omega\), then \(\mathbb{Z}\Omega=\bigoplus_i\mathbb{Z}\Omega_i\).  Furthermore, if \(\alpha_i=\sum\Omega_i\), then \(\alpha=\sum_i\alpha_i\), so that, if \(\iota_i\) is the morphism \(\mathbb{Z}\rightarrow\mathbb{Z}\Omega_i:1\mapsto\alpha_i\), we have the decomposition $$\iota=\sum\limits_{i=1}^m\iota_i.$$  Therefore, letting \(N_i\) be the \(K_\nu\)-scheme corresponding to \(\iota_i\), we have 
    $$N(K_\nu)=\prod\limits_{i=1}^mN_i(K_\nu).$$
    It thus remains to determine \(N_i(K_\nu)\).  Now, it is well known that each \(g_i\) corresponds to a place \(\eta_i\) of \(K'\) over \(\nu\).  More explicitly, we can let \(\omega_i=g_i^{-1}\omega\) and take \(\eta_i\) to be the place of \(K'\) divided by \(\omega_i\), so that \(g_i^{-1}\mathcal{D}_{\omega/\nu}g_i=\mathcal{D}_{\omega_i/\nu}\) and \(\mathcal{D}_{\omega_i/\eta_i}=\mathcal{D}_{\omega_i/\nu}\cap G_{K'}\).  Letting \(\Omega_i'=\mathcal{D}_{\omega_i/\nu}/\mathcal{D}_{\omega_i/\eta_i}\), observe that the maps $$\mathcal{D}_{\omega/\nu}\rightarrow\mathcal{D}_{\omega_i/\nu}:\gamma\mapsto \gamma^{g_i}$$ $$\Omega_i\rightarrow\Omega_i':\gamma\overline{g}_i\mapsto \overline{\gamma^{g_i}}$$ 
    determine an isomorphism of Galois representations \((\mathcal{D}_{\omega/\nu},\Omega_i)\rightarrow(\mathcal{D}_{\omega_i/\nu},\Omega_i')\) allowing an identification \(N_i(K_\nu)=N_i'(K_\nu)\), where \(N_i'\) is the \(K_\nu\)-scheme corresponding to the morphism \(\mathbb{Z}\rightarrow\mathbb{Z}\Omega_i':1\mapsto\sum\Omega_i'\).  By Lemma \ref{norm}, we know that \(N_i'(K_\nu)\) is the field norm \((K_{\eta_i}')^\times\rightarrow K_\nu^\times\). It follows that \(N(K_\nu)\) is the component over \(\nu\) of the idele norm \(\mathbb{I}_{K'}\rightarrow\mathbb{I}_K\).
\end{proof}

\section{Corresponding Abelian Extensions}\label{cae}

Let \(K_1\) and \(K_2\) be number fields integrally equivalent over \(F\).  Writing \(T_i=\)Res\(_{K_i/F}(\mathbb{G}_m)\), the isomorphism \(\mathbb{Z}[G/G_{K_1}]\simeq\mathbb{Z}[G/G_{K_2}]\) gives an isomorphism \(T_1\simeq T_2\), by Lemma \ref{chargroup}.  The \(F\)-points of the \(T_i\) are therefore isomorphic \(K_1^\times \simeq K_2^\times\), as are the \(\mathbb{A}_F\)-points \(\mathbb{I}_{K_1}\simeq\mathbb{I}_{K_2}\), where \(\mathbb{A}_F\) denotes the ring of \(F\)-adeles.  Furthermore, the isomorphism of ideles respects the diagonal embeddings \(K_i^\times\rightarrow\mathbb{I}_{K_i}\), so we can quotient to get an isomorphism of idele class groups \(\varphi:C_{K_1}\rightarrow C_{K_2}\) \cite{prasad_refined_2017}.  By class field theory, a finite abelian extension \(L_i/K_i\) is uniquely determined by the finite-index open subgroup of \(C_{K_i}\) given by the image of the idele class norm \(N_{L_i/K_i}\).  Moreover, every finite-index open subgroup of \(C_{K_i}\) is such an image: $$L_i/K_i\leftrightarrow N_{L_i/K_i}(C_{L_i})\subset C_{K_i}.$$
We say that \(L_1/K_1\) and \(L_2/K_2\) \textbf{correspond} if \(N_{L_2/K_2}(C_{L_2})=\varphi(N_{L_1/K_1}(C_{L_1}))\).  Observe that \(\varphi\) induces an isomorphism of Galois groups \(\text{Gal}(L_1/K_1)\rightarrow\text{Gal}(L_2/K_2)\), so that, in particular, \([L_1:K_1]=[L_2:K_2]\), which implies \([L_1:F]=[L_2:F]\).  We state the following proposition without proof, as it is an immediate consequence of our definition of corresponding abelian extensions.
\begin{proposition}\label{unandint}
Suppose \(L_i/K_i\), \(i=1,2\) correspond, and \(L_i'/K_i,i=1,2\) correspond.  Then \(L_iL_i'/K_i\), \(i=1,2\) correspond and \(L_i\cap L_i'/K_i\), \(i=1,2\) correspond.  
\end{proposition}
\noindent The remainder of this section is devoted to fleshing out the relations between corresponding \(L_1\) and \(L_2\).  Specifically, in Section \ref{similar}, we prove Theorem \ref{P}, and then in Section \ref{different}, we demonstrate that the \(L_i\) are not even weakly Kronecker equivalent in general.  Throughout the sequel, \(L_i/K_i\), \(i=1,2\) will denote corresponding abelian extensions of integrally equivalent number fields.  We will also freely identify \(\prod K_{i,\nu}^\times\) (product over some finite set of places) with its image under the projection \(\mathbb{I}_{K_i}\rightarrow C_{K_i}\).  Whether we are working in the idele group or the idele class group should be clear from context.  Lastly, given an abelian extension \(L_i/K_i\) and a place \(\nu\) of \(K_i\), we will refer to the completion of \(L_i\) at a place over \(\nu\) simply as \(L_{i,\nu}\), since \(L_{i,\eta_1}\simeq L_{i,\eta_2}\) for places \(\eta_1,\eta_2\) of \(L_i\) over \(\nu\).
\subsection{Arithmetic Similarity}\label{similar}\hfill\\\\
\indent Given a place \(\omega\) of \(F\), all direct products from here on are over the places \(\nu\) (or \(\eta\)) of \(K_i\) dividing \(\omega\), unless otherwise stated.  We know that
\(T_i(F_\omega)=\prod K_{i,\nu}^\times\), so applying \(T_1\simeq T_2\) to \(F_\omega\) gives an isomorphism
\begin{equation}\label{local}
    \varphi_\omega: \prod\limits_{\nu|\omega}K_{1,\nu}^\times\rightarrow\prod\limits_{\eta|\omega}K_{2,\eta}^\times.
\end{equation}
Define \(N_{i,\omega}\subset C_{K_i}\) by \(N_{i,\omega}=N_{L_i/K_i}(C_{L_i})\cap\prod K_{i,\nu}^\times\).  Recall that 
\begin{equation}\label{localgroups}
K_{i,\nu}^\times/(N_{i,\omega}\cap K_{i,\nu}^\times)\simeq\text{Gal}(L_{i,\nu}/K_{i,\nu})
\end{equation}
and 
\begin{equation}\label{ram}
|\mathcal{O}_{K_{i,\nu}}^\times/(N_{i,\omega}\cap\mathcal{O}_{K_{i,\nu}}^\times)|=e(L_{i,\nu}/K_{i,\nu}).
\end{equation}
See Chapter X in \cite{tate_class_2008} and Chapter V in \cite{serre_local_1979} for details.  The functions \(\varphi_\omega\) provide a means of probing the \(L_i\) locally.  This proves fruitful,  allowing us to compare, respectively, \(\zeta_{L_i}\), \(\Delta_{L_i}\), and the Galois closure of \(L_i/\mathbb{Q}\), \(i=1,2\). Then, in Theorem \ref{thanks} we will make use of the Norm Limitation Theorem from global class field theory to conclude that corresponding extensions have the same maximal abelian sub-extension over \(F\), which allows us to relate their \(K\)-groups.  We begin with an observation shedding light on both the Galois closures and the \(\zeta_{L_i}\).

\begin{proposition}\label{char2}\label{galoisclosure}
    If \(K_1,K_2\) are integrally equivalent number fields over \(F\), with corresponding abelian extensions \(L_i/K_i\), \(i=1,2\), then \(L_1/F\) and \(L_2/F\) have the same Galois closure.  
\end{proposition}
\begin{proof}
Recall that the Galois closure of an extension of number fields \(E/F\) is uniquely determined by the set of primes in \(F\) which split completely in \(E\).  Thus, it suffices to show that a prime \(\omega\) of \(F\) splits completely in \(L_1\) if and only if it does so in \(L_2\).  Now, \(\omega\) splits completely in \(L_i\) if and only if it splits completely in \(K_i\) and Gal\((L_{i,\nu}/K_{i,\nu})\) is trivial for each place \(\nu\) of \(K_i\) over \(\omega\). But then, since \(K_1\) and \(K_2\) are integrally equivalent over \(F\), we know that \(\omega\) splits completely in \(K_1\) if and only if it does so in \(K_2\).  Furthermore, by equation \ref{localgroups}, Gal\((L_{i,\nu}/K_{i,\nu})=1\) for every \(\nu\) over \(\omega\) if and only if \(N_{i,\omega}=\prod K_{i,\nu}^\times\).  Since \(\varphi_\omega\) in equation \ref{local} restricts to an isomorphism \(N_{1,\omega}\rightarrow N_{2,\omega}\), we know that \(N_{1,\omega}=\prod K_{1,\nu}^\times\) if and only if \(N_{2,\omega}=\prod K_{2,\eta}^\times\).  Thus \(\omega\) splits completely in \(L_1\) if and only if it does so in \(L_2\), so the \(L_i\) indeed have the same Galois closure over \(F\). 
\end{proof}

\begin{corollary}
    Corresponding abelian extensions are ultra-coarsely arithmetically equivalent in the sense of Proposition \ref{ultra-coarse}.
\end{corollary}

\noindent Before stating and proving the next corollary, we fix some terminology.  Say that a rational prime \(p\) \textbf{splits relatively completely} in \(L_i/K_i\) if every prime of \(K_i\) over \(p\) splits completely in \(L_i\), and observe that arguments in the proof of Proposition \ref{galoisclosure} guarantee that a rational prime splits relatively completely in \(L_1/K_1\) if and only if it does so in \(L_2/K_2\).

\begin{corollary}\label{zeta}
    If \(\zeta_{L_i}(s)=\sum_{n\in\mathbb{N}_+}\frac{a_i(n)}{n^s}\), and \(n\) is an integer whose prime divisors all split relatively completely in the \(L_i/K_i\), \(i=1,2\), then \(a_1(n)=a_2(n)\).
\end{corollary}
\begin{proof}
    Let \(\mathcal{P}\) consist of those rational primes which split relatively completely in the \(L_i/K_i\), \(i=1,2\). For \(\text{Re}(s)>1\), let $$\zeta_{L_i,\mathcal{P}}(s)=\prod\limits_{\mathfrak{p}|p\in\mathcal{P}}(1-N(\mathfrak{p})^{-s})^{-1},$$where the primes \(\mathfrak{p}\) are in \(\mathcal{O}_{L_i}\). A positive integer \(f\) is in \(S_{L_i}(p)\) with multiplicity equal to the number of prime ideals in \(\mathcal{O}_{L_i}\) with absolute norm \(p^f\), and \(\zeta_{L_i,\mathcal{P}}\) counts the number of ideals in \(\mathcal{O}_{L_i}\) whose absolute norm is only divisible by primes in \(\mathcal{P}\). For \(p\in\mathcal{P}\), \(S_{L_1}(p)=S_{L_2}(p)\), since \(S_{K_1}(p)=S_{K_2}(p)\) and \(p\) splits relatively completely in \(L_i/K_i\), \(i=1,2\).  Therefore, we have \(\zeta_{L_1,\mathcal{P}}=\zeta_{L_2,\mathcal{P}}\).  Now, if every prime divisor of \(n\) is in \(\mathcal{P}\), then \(a_i(n)\) is the coefficient of \(n^{-s}\) in \(\zeta_{L_i,\mathcal{P}}\), and we conclude \(a_1(n)=a_2(n)\).
\end{proof}

\begin{remark}By Grunwald-Wang, given a finite set \(S\) of rational primes, the \(L_i/K_i\) can be chosen to split relatively completely over each \(p\in S\).  In particular, given \(M\in\mathbb{N}_+\), there are corresponding abelian extensions \(L_1,L_2\) such that \(a_1(n)=a_2(n)\) for any \(n\in\mathbb{N}_+\) with \(n\leq M\).
\end{remark}

\noindent Arithmetically equivalent \(K_1,K_2\) have the same signature and contain the same roots of unity, so there is an abstract isomorphism \(\mathcal{O}_{K_1}^\times\simeq\mathcal{O}_{K_2}^\times\).  For integrally equivalent \(K_1,K_2\), it turns out that the isomorphism \(\varphi:K_1^\times\rightarrow K_2^\times\) in fact restricts to an isomorphism \(\mathcal{O}_{K_1}^\times\simeq\mathcal{O}_{K_2}^\times\).  The keystone is Proposition \ref{integrality}, which may also be leveraged to relate the discriminants of the \(L_i\). 

\begin{proposition}\label{integrality}
    \(\varphi_\omega\) restricts to an isomorphism \(\prod\limits \mathcal{O}_{1,\nu}^\times\simeq\prod\limits \mathcal{O}_{2,\eta}^\times\). 
\end{proposition}
\begin{proof}
    Let \(n\) be the number of distinct prime divisors of \(\omega\) in \(K_i\) (note this is independent of whether \(i=1\) or \(i=2\)), \(H_i=\prod\limits K_{i,\nu}^\times\), \(U_i=\prod\limits \mathcal{O}_{i,\nu}^\times\). We know \(H_i\simeq \mathbb{Z}^n\bigoplus U_i\), so \(H_1/U_1\simeq \mathbb{Z}^n\), and therefore \(H_2/\varphi_\omega(U_1)\simeq\mathbb{Z}^n\). But then, letting \(\pi\) denote the projection \(H_2\rightarrow H_2/\varphi_\omega(U_1)\), we know that \(\pi(H_2)\simeq\pi(\mathbb{Z}^n)\pi(U_2)\) is free abelian of rank \(n\) and \(\pi(U_2)\simeq U_2/U_2\cap\varphi_\omega(U_1)\). Now, the \(U_i\) are virtually pro-\(p\), so they have no nontrivial free abelian quotients.  Therefore, \(\pi(U_2)=0\), so \(U_2\subset\varphi_\omega(U_1)\). The inclusion \(\varphi_\omega(U_1)\subset U_2\) is a consequence of the decomposition of \(H_2\) above, because otherwise \(U_1\) would surject a nontrivial free abelian group.
\end{proof}
\begin{corollary}
    The isomorphism \(K_1^\times\simeq K_2^\times\) restricts to an isomorphism \(\mathcal{O}_{K_1}^\times\simeq\mathcal{O}_{K_2}^\times\). 
\end{corollary}
\begin{proof}
    Follows from Proposition \ref{integrality}, along with the facts that \(\mathcal{O}_{K_i}^\times\) is precisely the set of elements of \(K_i^\times\) with \(\nu\)-adic valuation equal to \(0\) for every place \(\nu\) of \(K_i\), and \(K_i^\times\cap\mathcal{O}_{K_{i,\nu}}^\times\) is precisely the set of elements of \(K_i^\times\) with \(\nu\)-adic valuation equal to \(0\).  
\end{proof}
\begin{corollary}\label{discrdiv}
    A rational prime divides \(\Delta_{L_1}\) if and only if it divides \(\Delta_{L_2}\).  
\end{corollary}
\begin{proof}
    Note by equation \ref{ram} that a rational prime \(p\) is unramified in \(L_i\) if and only if it is unramified in \(K_i\) and \(U_{i,p}\subset N_{i,p}\), where \(U_{i,p}=\prod_{\nu|p}\mathcal{O}^\times_{K_{i,\nu}}\), and \(N_{i,p}=N_{L_i/K_i}(C_{L_i})\cap\prod_{\nu|p}K_{i,\nu}^\times\).  But then the \(K_i\) are arithmetically equivalent, so they are unramified over the same rational primes, and Proposition \ref{integrality} says that \(U_{1,p}\subset N_{1,p}\) if and only if \(U_{2,p}\subset N_{2,p}\).  Hence, a rational prime is unramified in \(L_1\) if and only if it is unramified in \(L_2\).  Since the prime divisors of \(\Delta_{L_i}\) are exactly those rational primes which ramify in \(L_i\), the claim holds.
\end{proof}

\noindent We have now come to the main theorem of this section, whose proof requires Theorem \ref{idelenorm} and class field theory.  Theorem \ref{thanks} is used to relate the odd \(K\)-groups of corresponding abelian extensions.

\begin{theorem}\label{thanks}
    If \(N_{L/K}\) denotes the idele class norm \(C_L\rightarrow C_K\), then \(N_{L_1/F}(C_{L_1})=N_{L_2/F}(C_{L_2})\).  In particular, \(L_1/F\) and \(L_2/F\) contain the same maximal abelian subextension.
\end{theorem}
\begin{proof}
    Let \(\Omega_i=G/G_{K_i}\), \(A\) be the \(G\)-equivariant linear isomorphism \(\mathbb{Z}\Omega_1\rightarrow\mathbb{Z}\Omega_2\), and define $$\alpha_i=\sum\Omega_i.$$
    Since the action of \(G\) on \(\Omega_i\) is transitive, \(\mathbb{Z}\alpha_i\subset\mathbb{Z}\Omega_i\) is the only rank 1 submodule fixed pointwise by \(G\).  But then \(gA\alpha_1=Ag\alpha_1=A\alpha_1\) for each \(g\in G\), so \(A\alpha_1=n\alpha_2\) for some \(n\in\mathbb{Z}\).  Notice that \(\alpha_i\) corresponds to the vector \((1,1,...,1)\) when we use the cosets as a basis for \(\mathbb{Z}\Omega_i\), so in fact \(n\) is an eigenvalue of \(A\).  Since \(A\) and \(A^{-1}\) are each represented in this basis by an invertible integer matrix, \(n=\pm 1\).  Letting \(N_i:T_i\rightarrow\mathbb{G}_m\) be the scheme corresponding to the \(G\)-module morphism \(\mathbb{Z}\rightarrow\mathbb{Z}\Omega_i:1\mapsto\alpha_i\), we therefore have the commutative diagram of schemes
    
$$\begin{CD}
        T_1 @>>> T_2\\
        @VN_1VV     @VVN_2V\\
\mathbb{G}_m @>>> \mathbb{G}_m
    \end{CD}$$

\noindent where the bottom isomorphism is either the identity or inversion.  Taking \(\mathbb{A}_F\)-points gives
$$\begin{CD}
        \mathbb{I}_{K_1} @>>> \mathbb{I}_{K_2}\\
        @VVV     @VVV\\
\mathbb{I}_F @>>> \mathbb{I}_F
    \end{CD}$$

\noindent where the vertical arrows are the idele norms, by Theorem \ref{idelenorm}.  The diagonal embeddings are respected, so we can pass to idele class groups and obtain the following commutative diagram.

\begin{equation}\label{hal}
    \begin{CD}
        C_{K_1} @>\varphi>> C_{K_2}\\
        @VN_{K_1/F}VV     @VVN_{K_2/F}V\\
        C_F @>>> C_F
    \end{CD}
\end{equation}

\noindent From the fact that \(N_{L_i/F}=N_{K_i/F}\circ N_{L_i/K_i}\) and the commutativity of diagram (\ref{hal}), we find that \(N_{L_1/F}(C_{L_1})=N_{L_2/F}(C_{L_2})\), so  \(L_1/F\) and \(L_2/F\) have the same maximal abelian subextension, by the Norm Limitation Theorem (c.f. Theorem 7.3.10 in \cite{kedlaya_notes_nodate}, Theorem 7 in Chapter XIV of \cite{tate_class_2008}).    
\end{proof}

\begin{corollary}\label{primabext}
    Fixing a separable closure \(F^s\) of \(F\), if \(\alpha\in F^s\), and \(F(\alpha)\) is abelian, then \(K_1(\alpha)\) and \(K_2(\alpha)\) correspond.  
\end{corollary}
\begin{proof}
    Suppose \(L_1=K_1(\alpha)\).  By Theorem \ref{thanks}, \(L_1/F\) and \(L_2/F\) have the same maximal abelian subextension, so \(\alpha\in L_2\).  Hence \([K_2(\alpha):K_2]\leq [L_2:K_2]\).  But then, \([L_2:K_2]=[K_1(\alpha):K_1]\), so \([K_2(\alpha):K_2]\leq[K_1(\alpha):K_1]\).  We can just as easily argue that \([K_1(\alpha):K_1]\leq[K_2(\alpha):K_2]\), so in fact \([K_1(\alpha):K_1]=[K_2(\alpha):K_2]\), and thus \([K_2(\alpha):K_2]=[L_2:K_2]\), so \(L_2=K_2(\alpha)\).  By symmetry we conclude that \(L_1=K_1(\alpha)\) if and only if \(L_2=K_2(\alpha)\).  In other words, \(K_1(\alpha)\) and \(K_2(\alpha)\) are corresponding abelian extensions.
\end{proof}

\begin{proposition}\label{sig}
    If \([L_i:K_i]\) is odd, then \(L_1\) and \(L_2\) have the same signature. 
\end{proposition}
\begin{proof}
    Let \(F_\infty=\prod_{\nu|\infty}F_\nu\subset\mathbb{A}_F\) and set \(T_{i,\infty}=T_i(F_\infty)\) and \(N_{i,\infty}=N_{L_i/K_i}(C_{L_i})\cap T_{i,\infty}\). If \([L_i:K_i]\) is odd, then the image of \(T_{i,\infty}\) under the map \(C_{K_i}\rightarrow\Gal(L_i/K_i)\) is trivial, as \(T_{i,\infty}/N_{i,\infty}\) is a \(2\)-group.  Therefore, \(N_{i,\infty}=T_{i,\infty}\).  In particular, no archimedean prime of \(K_i\) ramifies in \(L_i\).  Since \(K_1\) and \(K_2\) have the same signature, so do \(L_1\) and \(L_2\).\\\\
    Alternatively, let \(G\) be the group of the Galois closure of the \(L_i\) over \(\mathbb{Q}\), and let \(\mathcal{D}\subset G\) be a subgroup of order \(1\) or \(2\).  Because \(K_1\) and \(K_2\) are integrally equivalent over \(F\), they are arithmetically equivalent, so by Proposition 2.6 in \cite{sutherland_stronger_nodate}, $$|\{g\mathcal{D}g^{-1}\subset G_{K_1}\text{ }|\text{ }g\in G\}|=|\{g\mathcal{D}g^{-1}\subset G_{K_2}\text{ }|\text{ }g\in G\}|.$$ But then the image of \(\mathcal{D}\cap G_{K_i}\) is trivial under the projection \(G_{K_i}\rightarrow G_{K_i}/G_{L_i}\), so \(\mathcal{D}\subset G_{K_i}\) if and only if \(\mathcal{D}\subset G_{L_i}\), and therefore $$|\{g\mathcal{D}g^{-1}\subset G_{L_1}\text{ }|\text{ }g\in G\}|=|\{g\mathcal{D}g^{-1}\subset G_{L_2}\text{ }|\text{ }g\in G\}|.$$  By Proposition 2.2 in \cite{sutherland_stronger_nodate}, it follows that \(G/G_{L_1}\simeq G/G_{L_2}\) as \(\mathcal{D}\)-sets.  This is then true, in particular, for \(\mathcal{D}\) a decomposition group of the archimedean prime. 
\end{proof}

\begin{proposition}
    If \(K_{n,t}(F)\) denotes the torsion part of \(K_n(F)\) for a number field \(F\), then for odd \(n\not\equiv 3\) mod \(8\), \(K_{n,t}(L_1)\simeq K_{n,t}(L_2)\).  Furthermore, if \([L_i:K_i]\) is odd, then \(K_n(L_1)\simeq K_n(L_2)\) for odd \(n\geq 3\).
\end{proposition}
\begin{proof}
   Let \(L_i/K_i\), \(i=1,2\), be corresponding abelian extensions and \(n\) an odd integer.  If \(n=1\), then we know the first claim holds, since the \(L_i\) have the same roots of unity, by Theorem \ref{thanks}, so assume additionally that \(n\geq 3\).  Letting \(\mathcal{C}_{i,p}=\{(\text{Gal}(L_i(\zeta_{p^k})/L_i),k)\}_{k\in\mathbb{N}}\), where \(\zeta_{p^k}\) is a primitive \(p^k\)-th root of unity, we know by equation \ref{exp} that the torsion part of \(K_n(L_i)\) is determined by \(\mathcal{C}_{i,p}\) for odd \(n\) not congruent to \(3\) mod \(8\). From Proposition \ref{unandint} and Corollary \ref{primabext}, we know that \(L_1(\zeta)/K_1\) and \(L_2(\zeta)/K_2\) correspond for any primitive root of unity \(\zeta\), since \(L_i(\zeta)=L_iK_i(\zeta)\).  Therefore, \(\text{Gal}(L_1(\zeta)/K_1)\simeq\text{Gal}(L_2(\zeta)/K_2)\), and this isomorphism restricts to an isomorphism \(\text{Gal}(L_1(\zeta)/L_1)\simeq\text{Gal}(L_2(\zeta)/L_2)\), by the definition of corresponding abelian extensions.  Therefore, \(\mathcal{C}_{1,p}=\mathcal{C}_{2,p}\) for each prime \(p\).  This verifies that \(K_{n,t}(L_1)\simeq K_{n,t}(L_2)\) for \(n\not\equiv 3\) mod \(8\).  Supposing now that \([L_i:K_i]\) is odd, then by Proposition \ref{sig}, the signatures of the \(L_i\) are the same, so \(K_n(L_1)\simeq K_n(L_2)\) for any odd \(n\geq 3\).
\end{proof}

\subsection{Corresponding Abelian Extensions That Are Not Weakly Kronecker Equivalent}\label{different}\hfill\\\\
\indent We prove that non-isomorphic integrally equivalent number fields always possess corresponding abelian extensions that are not weakly Kronecker equivalent.  Before proving the main result of this subsection, we need to set the stage with some lemmas.  We use \(e_i\) to denote the \(i^{th}\) standard basis vector.

\begin{lemma}\label{A}
    If \((G,G_1,G_2)\) is a non-trivial integral Gassmann triple\footnotemark\footnotetext{i.e. the \(G_i\) are not conjugate in \(G\)} of finite groups, and \(A\in\text{GL}_n(\mathbb{Z})\) is the matrix expressing the \(\mathbb{Z}[G]\)-module isomorphism \(\mathbb{Z}[G/G_1]\rightarrow\mathbb{Z}[G/G_2]\) using cosets as bases, then each row of \(A\) contains multiple non-zero components. 
\end{lemma}
\begin{proof}
    Since the groups are finite and \(G_1,G_2\) are not conjugate in \(G\), we know that for every \(g\in G\), there is \(h\in gG_2g^{-1}\) such that \(h\notin G_1\).  In particular, letting \(e_1^i\) be the basis element corresponding to the coset \(G_i\in G/G_i\) and \(g_1,...,g_n\) be a complete set of coset representatives for \(G/G_2\), then if \(Ae_1^1 = \sum_ia_ig_ie_1^2\) and \(a_j\neq 0\), there is \(g\in g_jG_2g_j^{-1}\) such that \(g\notin G_1\), so that \(ge_1^1\neq e_1^1\) and \(A(ge_1^1)=gA(e_1^1) = a_jg_je_1^2+\sum_{i\neq j}a_igg_ie_1^2\). In particular, the \(j^{th}\) row of \(A\) has at least two non-zero components.  By \(G\)-equivariance, we conclude that any row with a non-zero component must have multiple nonzero components.  But then \(A\) is invertible, so every row has at least one non-zero component.
   
\end{proof}

\begin{lemma}\label{mnsidentification}
Call a sublattice of the standard lattice in Euclidean space \textbf{normal} if it has a basis consisting of integral multiples of the standard basis.  The maximal normal sublattice \(\mathcal{L}\) of the lattice given by the \(\mathbb{Z}\)-span of the columns of \(M\in\text{GL}(r,\mathbb{Q})\cap M(r,\mathbb{Z})\) is \(\bigoplus_im_i\mathbb{Z}\), where \(m_i\) it the least positive integer \(m\) such that \(mM^{-1}e_i\in\mathbb{Z}^r\).
\end{lemma}
\begin{proof}
    The \(i^{th}\) component of any element of \(\mathcal{L}\) is an integer multiple of the least positive \(m\) such that \(me_i\in\mathcal{L}\), since \(\mathcal{L}\) is a normal sublattice.  However, \(me_i\in\mathcal{L}\) if and only if \(mM^{-1}e_i\in\mathbb{Z}^r\).
\end{proof}  

\begin{lemma}\label{ZnIso}
    Let \(p\) be a rational prime that splits completely in integrally equivalent number fields \(K_i\), \(i=1,2\), and set \(\mathcal{L}_{i,p}=T_i(\mathbb{Q}_p)/U_{i,p}\).  Identifying \(\mathcal{L}_{i,p}\) with \(\mathbb{Z}^n\), the automorphism \(\mathbb{Z}^n\rightarrow\mathbb{Z}^n\) determined by the linear isomorphism \(\mathcal{L}_{1,p}\rightarrow\mathcal{L}_{2,p}\) induced by \(\varphi_p\) in equation \ref{local} is precisely the transpose of the matrix expressing the \(G\)-equivariant linear isomorphism \(A:\mathbb{Z}[ G/G_{K_2}]\rightarrow\mathbb{Z}[G/G_{K_1}]\) with bases given by the cosets \(G/G_{K_i}\).
\end{lemma}
\begin{proof}
A map of sets \(s:S_1\rightarrow S_2\) determines an extension to a linear map \(\mathbb{Q}_p[S_1]\rightarrow\mathbb{Q}_p[S_2]\), which we also refer to as \(s\).  Let \(X_i\) be the character group of \(T_i\), and suppose \(g_{i1},...,g_{in}\in G\) is a full set of representatives for \(G/G_i\), with \(x_{ij}\in X_i\) the preimage of \(g_{ij}G_i\) under the isomorphism \(\varphi_i:X_i\rightarrow\mathbb{Z}[G/G_i]\).  Since \(p\) splits completely in \(K_i\), we know \(T_i(\mathbb{Q}_p)=\text{Hom}(\mathbb{Q}_p[X_i],\mathbb{Q}_p)\), and a morphism \(f:\mathbb{Q}_p[X_i]\rightarrow\mathbb{Q}_p\) is uniquely determined by the \(n\)-tuple \((f(x_{i1}),...,f(x_{in}))\in(\mathbb{Q}_p^\times)^n\).  Denote by \(\varphi\) the isomorphism \(\text{Hom}(\mathbb{Q}_p[X_1],\mathbb{Q}_p)\rightarrow\text{Hom}(\mathbb{Q}_p[X_2],\mathbb{Q}_p)\), and let \(A':X_2\rightarrow X_1\) be given by \(A'=\varphi_1^{-1}\circ A\circ\varphi_2\).  Observe that \(\varphi(f)=f\circ A'\), so that, if \(f\in\text{Hom}(\mathbb{Q}_p[X_1],\mathbb{Q}_p)\) is given by \(f(x_{1j})=\alpha_j\), then \(\varphi(f)(x_{2j})=(\alpha_1,...,\alpha_n)^{\text{col}_j(A)}\), where for vectors \(a=(a_1,...,a_n),b=(b_1,...,b_n)\), we are writing \(a^b=a_1^{b_1}\cdots a_n^{b_n}\).  Now, write \(\alpha_i=p^{m_i}\alpha_i'\), where \(\alpha_i'\) is a unit.  Passing to the quotient \((\mathbb{Q}_p^\times)^n/(\mathbb{Z}_p^\times)^n\simeq\mathbb{Z}^n\), we see that the image of \(\alpha=(\alpha_1,...,\alpha_n)\) is identified with \(m=(m_1,...,m_n)\), while the image of \((\alpha^{\text{col}_1(A)},...,\alpha^{\text{col}_n(A)})\) is identified with \(A^Tm\).
\end{proof}

\noindent Recall from Corollary \ref{discrdiv} that a rational prime \(p\) is unramified in \(L_1\) if and only if it is unramified in \(L_2\).  With \(N_{i,p}\) and \(U_{i,p}\) as in Corollary \ref{discrdiv}, let  \(\mathcal{L}_{i,p}'=N_{i,p}/U_{i,p}\) and \(\mathcal{L}_{i,p}=T_i(\mathbb{Q}_p)/U_{i,p}\).  Identifying \(\mathcal{L}_{i,p}\) with \(\mathbb{Z}^n\), view \(\mathcal{L}_{i,p}'\) as a sublattice.  Note that the orders of the relative decomposition groups for \(L_i/K_i\) over a rational prime \(p\) unramified in \(L_i\) are exactly the positive multiples of the standard basis vectors forming a basis of the largest normal sublattice in \(\mathcal{L}_{i,p}'\), by equation \ref{localgroups}.  Though \(\mathcal{L}_{i,p}'\) certainly can be a normal sublattice in \(\mathcal{L}_{i,p}\), this is by no means the case in general (c.f. \cite{tate_class_2008} Chapters VII and X).

\begin{theorem}\label{notwKeq}
     Given non-isomorphic integrally equivalent number fields, there exist corresponding abelian extensions which are not weakly Kronecker equivalent.
\end{theorem}
\begin{proof}
    Suppose \(K_1,K_2\) are non-isomorphic, integrally equivalent number fields whose Galois closure over \(\mathbb{Q}\) has group \(G\), and let \(A:\mathbb{Z}[G/G_{K_2}]\rightarrow\mathbb{Z}[G/G_{K_1}]\) be a \(G\)-equivariant isomorphism inducing a correspondence of the abelian extensions of the \(K_i\).  If \(p\) is an odd rational prime that splits completely in the \(K_i\), by a Grunwald-Wang argument (for instance, intersecting appropriate subgroups of \(C_{K_1}\) guaranteed to exist by Theorem 6  in Chapter X of \cite{tate_class_2008}), given a prime \(q\) coprime to every cofactor of \(A\), there is an abelian extension \(L_1\) of \(K_1\) in which \(p\) is unramified with \(\mathcal{L}_{1,p}'=\mathbb{Z}\bigoplus q\mathbb{Z}\bigoplus\cdots\bigoplus q\mathbb{Z}\).  In particular, \(\mathcal{L}_{1,p}'\) is equal to its maximal normal sublattice, and \(S_{L_1}(p)=\{\{1,q,...,q\}\}\).  Now, let \(M=A^T\text{diag}(1,q,...,q)\).  By Lemma \(\ref{ZnIso}\), we know that \(\mathcal{L}_{2,p}'\) is given by the \(\mathbb{Z}\)-span of the columns of \(M\), so the splitting type of \(p\) in \(L_2\) is given by the set of least positive integers \(m_i\) such that \(m_iM^{-1}e_i\) is integral, by Lemma \ref{mnsidentification}.  Since \(A\) is \(G\)-equivariant, so is \(A^{-1}\).  Furthermore, because \(K_1\not\simeq K_2\), every row of \(A^{-1}\) has multiple non-zero entries, by Lemma \ref{A}. If \(C\) denotes the cofactor matrix of \(A\), we know that \(A^{-1}=\pm C^T\), so every column of \(C\) has multiple non-zero entries.  Now, \(M^{-1}=\pm \text{diag}(1,q^{-1},...,q^{-1})C\), and every column of \(C\) has multiple non-zero entries, so \(m_iM^{-1}e_i\) is integral if and only if \(q|m_i\), by choice of \(q\), so the splitting type of \(p\) in \(L_2\) is \(\{\{q,...,q\}\}\).  Therefore, \(\gcd(S_{L_1}(p))=1\), while \(\gcd(S_{L_2}(p))=q\), so the \(L_i\) are not weakly Kronecker equivalent, by Theorem \ref{noexwKreq}.
\end{proof}

\begin{remark}
    We know that a non-trivial correspondence of abelian extensions cannot preserve arithmetic equivalence classes (c.f. Theorem \ref{solthm}).  We have also just seen that a correspondence of abelian extensions induced by a non-trivial integral equivalence of number fields cannot preserve weak Kronecker equivalence classes.  The question remains whether or not this is a general phenomenon.  In other words, letting \(\mathcal{K}(F)\) denote the weak Kronecker class of a number field \(F\), if \(K_1,K_2\) are number fields admitting a bijective correspondence of abelian extensions \(\varphi: L_1/K_1\mapsto \varphi(L_1/K_1)\) such that \(\mathcal{K}(\varphi(L_1))=\mathcal{K}(L_1)\) for every abelian extension \(L_1\) of \(K_1\), does it follow that \(K_1\simeq K_2\)?
\end{remark}

\section{Diagrams in Homology and Cohomology}\label{homology}
Given an integral Gassmann triple \((\Gamma,\Gamma_1,\Gamma_2)\) and a \(\Gamma\)-module \(A\), Arapura et. al proved that there are isomorphisms in cohomology \(H^*(\Gamma_1,A)\simeq H^*(\Gamma_2,A)\) compatible with corestriction to and restriction from \(H^*(\Gamma,A)\), where the \(\Gamma_i\)-module structure on \(A\) is \(\text{Res}_{\Gamma_i}^\Gamma\)(A)  \cite{arapura_integral_2019}.  We show that there are in fact compatible isomorphisms \(\varphi: H(\Gamma_1,A)\rightarrow H(\Gamma_2,A)\), where \(H\) denotes either \(H_*\) or \(H^*\).  Furthermore, for \(\Gamma_N\) a finite index normal subgroup of \(\Gamma\) contained in \(\Gamma_1\cap\Gamma_2\), and \(A\) a finite abelian group of order coprime to \([\Gamma:\Gamma_i]\), we demonstrate that \(\varphi\) is also compatible with restriction and corestriction to and from, respectively, \(H(\Gamma_N,A)\). If \(U_i\subset H_1(\Gamma_i)\), \(i=1,2\) are index \(t<\infty\) subgroups such that \(\varphi: H_1(\Gamma_1)\rightarrow H_1(\Gamma_2)\) restricts to an isomorphism \(U_1\rightarrow U_2\), we say that \(U_1\) and \(U_2\) \textbf{correspond}. Furthermore, if \(\Gamma_i'\) is the preimage of \(U_i\) in \(\Gamma_i\) under the natural projection \(\Gamma_i\rightarrow H_1(\Gamma_i)\), then we say \(\Gamma_1'\) and \(\Gamma_2'\) \textbf{correspond} if \(U_1\) and \(U_2\) do.  We show in Theorem \ref{gthm} that corresponding \(\Gamma_1'\) and \(\Gamma_2'\) have the same normal core in \(\Gamma\), assuming \(t\) is coprime to \([\Gamma:\Gamma_i]\).  Theorem \ref{g1} is then proved as a corollary.  In Proposition \ref{finalapp}, we conclude with an arithmetic application of the techniques in this section. Throughout, \([\Gamma:\Gamma_i]<\infty\), \(\bigotimes=\bigotimes_\mathbb{Z}\), and if coefficients are not specified for a (co)homology group, the coefficients are \(\mathbb{Z}\) with trivial action.
\begin{lemma}\label{comp}
    If \(\Omega_1,\Omega_2\) are finite and transitive \(G\)-sets, then an isomorphism \(\varphi:\mathbb{Z}\Omega_1\rightarrow\mathbb{Z}\Omega_2\) of \(\mathbb{Z}G\)-modules, or its negative, is natural with respect to both the diagonal embeddings \(\Delta_i:\mathbb{Z}\rightarrow\mathbb{Z}\Omega_i:n\mapsto n\sum\Omega_i\) and the augmentations \(\varepsilon_i:\mathbb{Z}\Omega_i\rightarrow\mathbb{Z}:\sum_{\omega\in\Omega_i}n_\omega\omega\mapsto\sum_{\omega\in\Omega_i}n_\omega\).  Precisely, either \(\varphi\circ\Delta_1=\Delta_2\) and \(\varepsilon_2\circ\varphi=\varepsilon_1\), or \((-\varphi)\circ\Delta_1=\Delta_2\) and \(\varepsilon_2\circ(-\varphi)=\varepsilon_1\).
\end{lemma}
\begin{proof}
    Let \(A\) be the matrix expressing the isomorphism \(\varphi\) with respect to the bases \(\Omega_i\). Observe as in the proof of Proposition \ref{thanks} that, by \(G\)-equivariance, \((1,...,1)\) is an eigenvector of \(A\), and therefore \(A(1,...,1)=\pm(1,...,1)\).  Furthermore, letting \(\mathbb{Z}\Omega_i^*\) denote the dual \(G\)-module to \(\mathbb{Z}\Omega_i\) consisting of all \(\mathbb{Z}\)-linear maps \(\mathbb{Z}\Omega_i\rightarrow\mathbb{Z}\), with \(G\)-action given by \(gf(\omega)=f(g^{-1}\omega)\), for \(f\in\mathbb{Z}\Omega_i^*, \omega\in\Omega_i\), \(g\in G\), then the induced \(\mathbb{Z}G\)-module isomorphism \(\varphi^*:\mathbb{Z}\Omega_2^*\rightarrow\mathbb{Z}\Omega_1^*\) is expressed by \(A^T\) with respect to the dual bases \(\Omega_i^*\).  We conclude that \((1,...,1)\) is an eigenvector of \(A^T\).  Observe that \((1,...,1)\in\mathbb{Z}\Omega_i^*\) is precisely \(\varepsilon_i\), so \(\varepsilon_2\circ \varphi=\pm\varepsilon_1\).  Finally, note that if \(\varphi\circ\Delta_1=c\Delta_2\), and \(\varepsilon_2\circ\varphi=c'\varepsilon_1\), then $$c|\Omega_2|=(\varepsilon_2\circ\varphi\circ\Delta_1)(1)=c'|\Omega_1|.$$
    Since \(|\Omega_1|=|\Omega_2|\), we conclude \(c=c'\).
\end{proof}
\begin{corollary}\label{blemma}
     If \((\Gamma,\Gamma_1,\Gamma_2)\) is an integral Gassmann triple, then, given a \(\Gamma\)-module \(A\), there is a commutative diagram
$$\label{hom}
    \begin{tikzcd}
      & H(\Gamma, A)
      \arrow[dr, bend left, "Res_{\Gamma_2}^\Gamma"]
      \arrow[dl, bend right, "Res_{\Gamma_1}^\Gamma"']
      & \\
       H(\Gamma_1, A)\arrow[ur, bend right, "Cor^\Gamma_{\Gamma_1}"]\arrow[rr, bend right, "\simeq"'] & 
          & H(\Gamma_2, A)\arrow[ul, bend left, "Cor^\Gamma_{\Gamma_2}"']\\
          &&
    \end{tikzcd}
$$
where \(H\) denotes either \(H^*\) or \(H_*\), and the \(\Gamma_i\)-module structure on \(A\) is \(\text{Res}_{\Gamma_i}^\Gamma\)(A).
\end{corollary}
\begin{proof}
    Letting \(\Omega_i=\Gamma/\Gamma_i\), recall that restriction \(\text{Res}^\Gamma_{\Gamma_i}\) (resp. corestriction \(\text{Cor}^\Gamma_{\Gamma_i}\)) in (co)homology is obtained by tensoring the diagonal embedding \(\Delta_i\) (resp. augmentation morphism \(\varepsilon_i\)) with \(A\), then applying \(H(\Gamma,-)\) and Shapiro's Lemma (c.f. \cite{brown_cohomology_1982} III.5.6 and III.9.A). By Lemma \ref{comp}, there is a commutative diagram of \(\mathbb{Z}\Gamma\)-module morphisms
    $$
    \begin{tikzcd}
      & \mathbb{Z}
      \arrow[dr, bend left, "\Delta_2"]
      \arrow[dl, bend right, "\Delta_1"']
      & \\
        \mathbb{Z}\Omega_1\arrow[ur, bend right, "\varepsilon_1"]\arrow[rr, bend right, "\simeq"'] & 
          & \mathbb{Z}\Omega_2\arrow[ul, bend left, "\varepsilon_2"']\\
          &&
    \end{tikzcd}
    $$
   Tensoring with \(A\) then applying \(H(\Gamma,-)\) and Shapiro's Lemma yields the desired result.
\end{proof}

\noindent From here on, \((\Gamma, \Gamma_1, \Gamma_2)\) is an integral Gassmann triple, \(m=[\Gamma:\Gamma_i]\), and \(A\) is a \(\Gamma\)-module, unless stated otherwise. We use \(\varphi\) to denote the isomorphism \(H(\Gamma_1,A)\rightarrow H(\Gamma_2,A)\) from Corollary \ref{blemma}.

\begin{lemma}\label{modq}
    Let \(A\) be a finite abelian group of order coprime to \(m\), and let \(\Gamma_N\) be a finite index normal subgroup of \(\Gamma\) contained in \(\Gamma_1\cap\Gamma_2\). There is a commutative diagram of the following form.
    $$
    \begin{tikzcd}
        H(\Gamma_1,A)\arrow[rr, bend left, "\varphi"]\arrow[dr, bend right, "Res_{\Gamma_N}^{\Gamma_1}"']&& H(\Gamma_2,A)\arrow[dl, bend left, "Res_{\Gamma_N}^{\Gamma_2}"]\\
        &H(\Gamma_N,A)\arrow[ul, bend right, "Cor^{\Gamma_1}_{\Gamma_N}"]\arrow[ur, bend left, "Cor^{\Gamma_2}_{\Gamma_N}"']&\\
        &&
    \end{tikzcd}
    $$
\end{lemma}
\begin{proof}
Because \(\Gamma_N\trianglelefteq \Gamma\), and \(\Gamma_N\subset\Gamma_i\), \(i=1,2\), the action of \(\Gamma_N\) on \(\Gamma_i\backslash\Gamma\) by right multiplication is trivial, so that \(|\Gamma_i\backslash\Gamma/\Gamma_N|=m\).  Now, we know that
$$\text{Res}^\Gamma_{\Gamma_i}\text{Cor}^\Gamma_{\Gamma_N}(z)=\sum\text{Cor}^{\Gamma_i}_{\Gamma_N}\text{Res}^{\Gamma_N}_{\Gamma_N}(gz),$$ where the sum is over representatives \(g\in\Gamma\) for the set of double cosets \(\Gamma_i\backslash\Gamma/\Gamma_N\) (c.f. Proposition 9.5 in Chapter III of \cite{brown_cohomology_1982}).  Clearly, \(\text{Res}^{\Gamma_N}_{\Gamma_N}(gz)=gz\), and, because \(\Gamma_N\trianglelefteq\Gamma\), we know that \(\text{Cor}^{\Gamma_i}_{\Gamma_N}(gz)=\text{Cor}^{\Gamma_i}_{\Gamma_N}(z)\), so \(\text{Res}^\Gamma_{\Gamma_i}\text{Cor}^\Gamma_{\Gamma_N}(gz)=m\text{Cor}^{\Gamma_i}_{\Gamma_N}(z)\).  Since \(\varphi\circ\text{Res}^{\Gamma}_{\Gamma_1}=\text{Res}^{\Gamma}_{\Gamma_2}\), we conclude \(m\varphi(\text{Cor}_{\Gamma_N}^{\Gamma_1}(z))=m\text{Cor}_{\Gamma_N}^{\Gamma_2}(z)\).  But then \(m\) is coprime to \(|A|\), so \(\varphi(\text{Cor}_{\Gamma_N}^{\Gamma_1}(z))=\text{Cor}_{\Gamma_N}^{\Gamma_2}(z)\), as desired.\\\\
Let \(z\in H(\Gamma_1,C)\), and recall that \(\text{Cor}_{\Gamma_1}^\Gamma(z)=\text{Cor}_{\Gamma_2}^\Gamma(\varphi(z))\), so that \begin{equation}\label{sheesh}
\text{Res}^\Gamma_{\Gamma_N}\text{Cor}_{\Gamma_1}^\Gamma(z)=\text{Res}^\Gamma_{\Gamma_N}\text{Cor}_{\Gamma_2}^\Gamma(\varphi(z)).
\end{equation}
Similarly as before, we know that 
$$\text{Res}^\Gamma_{\Gamma_N}\text{Cor}^\Gamma_{\Gamma_i}(z)=\sum\text{Cor}^{\Gamma_N}_{\Gamma_N}\text{Res}^{g\Gamma_ig^{-1}}_{\Gamma_N}(gz).$$ Now, \(\text{Res}^{g\Gamma_ig^{-1}}_{\Gamma_N}(gz)=\text{Res}^{\Gamma_i}_{\Gamma_N}(z),\) and \(\text{Cor}^{\Gamma_N}_{\Gamma_N}\) is the identity map, so we conclude \(\text{Res}^\Gamma_{\Gamma_N}\text{Cor}^\Gamma_{\Gamma_i}(z)=m\text{Res}^{\Gamma_i}_{\Gamma_N}(z)\).  Therefore, by equation \ref{sheesh}, we have \(m\text{Res}^{\Gamma_1}_{\Gamma_N}(z)=m\text{Res}^{\Gamma_2}_{\Gamma_N}(\varphi(z))\), and we see that \(\text{Res}^{\Gamma_1}_{\Gamma_N}(z)=\text{Res}^{\Gamma_2}_{\Gamma_N}(\varphi(z))\).
\end{proof}

\noindent Recall that restriction and corestriction are induced by chain maps. With the hypotheses of Lemma \ref{modq} and \(A\) an abelian group with trivial \(\Gamma\)-action, viewing \(H_*(\Gamma_i)\bigotimes A\) as a subgroup of \(H_*(\Gamma_i,A)\) as in the the universal coefficient theorem, the diagram
$$
    \begin{tikzcd}
        H_*(\Gamma_1)\bigotimes A\arrow[rr, bend left, "\varphi\otimes id_A"]\arrow[dr, bend right]&& H_*(\Gamma_2)\bigotimes A\arrow[dl, bend left]\\
        &H_*(\Gamma_N)\bigotimes A\arrow[ul, bend right]\arrow[ur, bend left]&\\
        &&
    \end{tikzcd}
$$
\noindent commutes, by Lemma \ref{modq} and naturality of the inclusion \(H_*(\Gamma_i)\bigotimes A\rightarrow H_*(\Gamma_i,A)\).  The upward maps are \(\text{Cor}^{\Gamma_i}_{\Gamma_N}\bigotimes \text{id}_{A}\), and the downward maps are \(\text{Res}^{\Gamma_i}_{\Gamma_N}\bigotimes \text{id}_{A}\). Recall that a finite abelian group is naturally a commutative ring with identity, and let \(A\) and \(B\) be abelian groups, with \(A\) finite. For us, the map \(B\rightarrow B\bigotimes A\) is always the abelian group homomorphism \(b\mapsto b\otimes 1\). We may now prove the main theorem of this section.
\begin{theorem}\label{gthm}
  If \(\Gamma_i'\subset\Gamma_i\), \(i=1,2\) are corresponding index \(t\) subgroups, then \(\Gamma_1'\) and \(\Gamma_2'\) have the same normal core in \(\Gamma\) if \((m,t)=1\).
\end{theorem}
\begin{proof}
    Suppose  \(\Gamma_i/\Gamma_i'\simeq A\) for \(A\) a finite abelian group of order coprime to \(m\), and fix projections \(\pi_i:H_1(\Gamma_i)\rightarrow A\) such that \(\pi_1=\pi_2\circ\varphi\) and \(\Gamma_i'\) is the preimage of \(\ker\pi_i\) under the natural map \(\Gamma_i\rightarrow H_1(\Gamma_i)\).
    Apply the functor \(-\bigotimes A\), viewing \(A\) as a trivial \(\Gamma\)-module, to the corresponding commutative diagram, and extend the result to the commutative diagram
    $$\label{gtohom}
    \begin{tikzcd}
        &\Gamma_1\arrow[r] & H_1(\Gamma_1)\bigotimes A\arrow[dd,"\varphi\otimes id_A"]  \arrow[dr]&\\
       N\arrow[ur, "\iota_1"]\arrow[dr, "\iota_2"']\arrow[r] &H_1(N)\bigotimes A\arrow[ur]\arrow[dr] &&A\bigotimes A \\
        &\Gamma_2\arrow[r] & H_1(\Gamma_2)\bigotimes A \arrow[ur] &\\
    \end{tikzcd}
    $$
    where \(\iota_i\) is inclusion, \(N\) is the normal core of \(\Gamma_i\) in \(\Gamma\), and the horizontal maps are of the form \(G\rightarrow H_1(G)\rightarrow H_1(G)\bigotimes A\).  The proof of Theorem 1 in \cite{perlis_equation_1977} guarantees that \(N\) is independent of choice of \(i=1,2\).  Since the map \(A\rightarrow A\bigotimes A:a\mapsto a\otimes 1\) is injective, and the diagram
    $$
    \begin{tikzcd}
    H_1(\Gamma_i)\arrow[d, "\pi_i"']\arrow[r] & H_1(\Gamma_i)\bigotimes A\arrow[d]\\
    A\arrow[r] & A\bigotimes A
    \end{tikzcd}
    $$
commutes for \(i=1,2\), the kernel of the map \(\Gamma_i\rightarrow A\bigotimes A\) is \(\Gamma_i'\), and we conclude  \(N\cap\Gamma_1'=N\cap\Gamma_2'\).  Now, if \(N_i\) is the normal core of \(\Gamma_i'\) in \(\Gamma\), we know that \(N_i\subset N\), so that \(N_i\) is equal to the normal core of \(N\cap\Gamma_i'\) in \(\Gamma\), and therefore \(N_1=N_2\).
\end{proof}
\begin{remark}
Letting \(\pi'_i\) be the projection \(\Gamma_i\rightarrow A\), note the argument above shows that in fact \(\pi'_1(n)=\pi'_2(n)\) for \(n\in N\) when \(|A|\) is coprime to \([\Gamma:\Gamma_i]\).  It is not clear to the author how one could draw a similar conclusion without the coprimality assumption.
\end{remark}
\noindent We can now easily prove Theorem \ref{g1}, which is a sort of geometric analog of Theorem \ref{P}.1.
\begin{proof}[\textbf{Proof of Theorem \ref{g1}}]
    Recall the usual identification \(H(\pi_1(M))=H(M)\) for a manifold with contractible universal cover, and observe that the map \(H_1(M')\rightarrow H_1(M)\) induced by a covering \(M'\rightarrow M\) is corestriction \(H_1(\pi_1(M'))\rightarrow H_1(\pi_1(M))\).  The subgroup of \(\pi_1(M)\) corresponding to the normal closure of \(M_i'\rightarrow M\) under the covering space Galois correspondence is precisely the normal core of \(\pi_1(M_i')\) in \(\pi_1(M)\).  The rest follows from Theorem \ref{gthm} and the fact that \((\pi_1(M), \pi_1(M_1), \pi_1(M_2))\) is an integral Gassmann triple.
\end{proof}
\noindent We conclude with an application to the arithmetic correspondence.  The proof is only sketched, as all the requisite ideas have been discussed in greater detail in this article.
\begin{proposition} \label{finalapp}
    If \(L_i/K_i\), \(i=1,2\) are corresponding abelian extensions of number fields \(K_1,K_2\) integrally equivalent over \(F\), letting \(K\) be the Galois closure of the \(K_i\) over \(F\), if \([L_i:K_i]\) is coprime to \([K_i:F]\), then \(L_1K=L_2K\).
\end{proposition}
\begin{proof}
    Let \(K/F\) be a Galois extension with group \(G\) and subextension \(K'/F\). If \(\Omega=G/G_{K'}\), and \(j\) is the scheme corresponding to the augmentation map \(\mathbb{Z}\Omega\rightarrow\mathbb{Z}\), one can show that \(j(\mathbb{A}_F)\), modulo principle ideles, gives the map \(C_F\rightarrow C_{K'}\) induced by inclusion \(F\rightarrow K'\). This observation, along with Theorem \ref{idelenorm} and Lemma \ref{comp}, reveals that for integrally equivalent number fields \(K_1\) and \(K_2\), there is an isomorphism \(C_{K_1}\rightarrow C_{K_2}\) such that following diagram commutes.
    $$
    
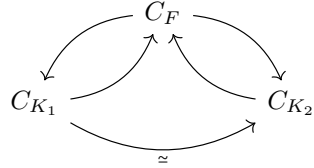
\begin{tikzcd}
      & C_F
      \arrow[dr, bend left]
      \arrow[dl, bend right]
      & \\
        C_{K_1}\arrow[ur, bend right]\arrow[rr, bend right, "\simeq"'] & 
          & C_{K_2}\arrow[ul, bend left]\\
          &&
    \end{tikzcd}
    $$
    where the upward arrows are idele class norms and the downward arrows are the maps induced by inclusion. Functoriality of the global Artin map (c.f. \cite{ramakrishnan_fourier_1999} Section 6.4) then guarantees commutativity of the following diagram.
    $$
    
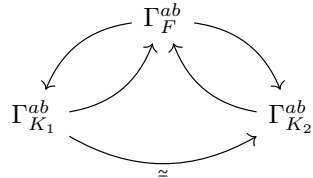
\begin{tikzcd}
      & \Gamma_F^{ab}
      \arrow[dr, bend left]
      \arrow[dl, bend right]
      & \\
        \Gamma_{K_1}^{ab}\arrow[ur, bend right]\arrow[rr, bend right, "\simeq"'] & 
          & \Gamma_{K_2}^{ab}\arrow[ul, bend left]\\
          &&
    \end{tikzcd}
    $$
    where, identifying abelianization with the functor \(H_1(-)\), the downward arrows are restriction and the upward arrows are corestriction in first homology.  From here, reason similarly as we have in this section to conclude that \(\Gamma_{L_1}\cap\Gamma_K=\Gamma_{L_2}\cap \Gamma_K\).
\end{proof}

\bibliography{cae2}
\bibliographystyle{alpha}

\end{document}